\documentclass[12pt,twoside]{amsart}
\usepackage{amssymb,amsmath,amsthm, amscd, enumerate, mathrsfs}
\usepackage{graphicx, hhline}
\usepackage[all]{xy}
\usepackage{color}
\usepackage[dvipdfmx]{hyperref}

\title{Nakai--Moishezon ampleness criterion for real line bundles}
\author{Osamu Fujino and Keisuke Miyamoto}
\date{2020/11/29, version 0.09}
\address{Osamu Fujino \\ Department 
of Mathematics, Graduate School of Science, 
Osaka University, Toyonaka, Osaka 560-0043, Japan}
\email{fujino@math.sci.osaka-u.ac.jp}
\address{Keisuke Miyamoto \\ Department of Mathematics, 
Graduate School of Science, 
Osaka University, Toyonaka, Osaka 560-0043, Japan}
\email{u901548b@ecs.osaka-u.ac.jp}

\DeclareMathOperator{\Supp}{Supp}
\DeclareMathOperator{\Bs}{Bs}
\DeclareMathOperator{\Exc}{Exc}
\DeclareMathOperator{\Pic}{Pic}
\DeclareMathOperator{\Div}{Div}

\newtheorem{thm}{Theorem}[section]
\newtheorem{lem}[thm]{Lemma}

\theoremstyle{definition}
\newtheorem{step}{Step}
\newtheorem{defn}[thm]{Definition}
\newtheorem*{ack}{Acknowledgments}  
\makeatletter
    
    \@addtoreset{equation}{section}
\makeatother

\begin{document}

\maketitle 

\begin{abstract} 
We show that the Nakai--Moishezon ampleness criterion holds for 
real line bundles on complete schemes. 
As applications, we treat the relative Nakai--Moishezon ampleness 
criterion for real line bundles and the Nakai--Moishezon ampleness 
criterion for real line bundles on complete algebraic spaces. 
The main ingredient of this paper is Birkar's characterization of 
augmented base loci of real divisors on projective schemes. 
\end{abstract}

\tableofcontents

\section{Introduction}\label{f-sec1}

Throughout this paper, a {\em{scheme}} means a separated scheme of 
finite type over an algebraically closed field $k$ of any 
characteristic. 
We call such a scheme a {\em{variety}} 
if it is reduced and irreducible. Let us start with 
the definition of $\mathbb R$-line bundles. 

\begin{defn}[$\mathbb R$-line bundles]\label{f-def1.1}
Let $X$ be a scheme (or an algebraic space). 
An {\em{$\mathbb R$-line bundle}} 
(resp.~a {\em{$\mathbb Q$-line bundle}}) is an element of 
$\Pic(X)\otimes _{\mathbb Z}\mathbb R$ 
(resp.~$\Pic(X)\otimes _{\mathbb Z} \mathbb Q$) 
where $\Pic(X)$ is the Picard group of $X$. 
\end{defn}

Similarly, we can define $\mathbb R$-Cartier divisors. 

\begin{defn}[$\mathbb R$-Cartier divisors]\label{f-def1.2}
Let $X$ be a scheme. An {\em{$\mathbb R$-Cartier divisor}} 
(resp.~a {\em{$\mathbb Q$-Cartier divisor}}) is an element of 
$\Div(X)\otimes _{\mathbb Z}\mathbb R$ 
(resp.~$\Div(X)\otimes _{\mathbb Z} \mathbb Q$) 
where $\Div(X)$ denotes the group of Cartier divisors on $X$. 
\end{defn}

We prove the Nakai--Moishezon ampleness criterion for 
$\mathbb R$-line bundles. The following 
theorem is the main result of this paper. 

\begin{thm}[Nakai--Moishezon ampleness criterion for 
real line bundles on complete schemes]\label{f-thm1.3}
Let $X$ be a complete scheme over an algebraically closed field and let 
$\mathcal L$ be an $\mathbb R$-line bundle 
on $X$. Then $\mathcal L$ is ample if and only if 
$\mathcal L^{\dim Z}\cdot Z>0$ for every 
positive-dimensional closed integral subscheme $Z \subset X$. 
\end{thm}

When $X$ is projective, Theorem \ref{f-thm1.3} 
is well known. 
It was first proved by Campana and 
Peternell (see \cite[1.3.~Theorem]{campana-peternell}). 
Then a somewhat simpler proof was given by Lazarsfeld 
in \cite[Theorem 2.3.18]{lazarsfeld}. 
Unfortunately, their arguments do not work for 
complete nonprojective schemes because they 
need an ample line bundle. Moreover, Kleiman's ampleness 
criterion does not always hold for complete nonprojective 
schemes (see 
\cite[Section 3]{fujino1} and 
\cite[Example 12.1]{fujino3}). 
Hence we need some new idea to prove Theorem \ref{f-thm1.3}. 
By the standard reduction argument, it is sufficient to 
treat the case where $X$ is a complete normal variety. 
Therefore, all we have to do is to establish the 
following theorem. 

\begin{thm}[Nakai--Moishezon ampleness criterion for 
real Cartier divisors on complete normal varieties]\label{f-thm1.4}
Let $X$ be a complete normal variety over an algebraically closed 
field and let $L$ be an $\mathbb R$-Cartier divisor 
on $X$. Then $L$ is ample if and only if 
$L^{\dim Z}\cdot Z>0$ for every positive-dimensional closed 
subvariety $Z\subset X$. 
\end{thm}

For the proof of Theorem \ref{f-thm1.4}, we use 
Birkar's characterization of augmented base loci of 
$\mathbb R$-divisors on projective 
schemes (see Theorem \ref{f-thm3.4}). 
Hence our approach is different from 
those of \cite{campana-peternell} and 
\cite{lazarsfeld}. Although we can not directly apply 
geometric arguments to $\mathbb R$-line bundles, 
we can generalize Theorem \ref{f-thm1.3} for 
proper morphisms. 

\begin{thm}[Relative Nakai--Moishezon ampleness criterion for 
real line bundles]\label{f-thm1.5}
Let $\pi\colon X\to S$ be a proper morphism between 
schemes and 
let $\mathcal L$ be an $\mathbb R$-line bundle 
on $X$. 
Then $\mathcal L$ is $\pi$-ample if and only if 
$\mathcal L^{\dim Z}\cdot Z>0$ for every 
positive-dimensional closed integral subscheme 
$Z\subset X$ such that $\pi(Z)$ is a point. 
\end{thm}

For the details of the Nakai--Moishezon ampleness 
criterion and Kleiman's ampleness criterion for 
line bundles, see \cite{kleiman}. 
The reader can find many nontrivial examples of 
complete nonprojective varieties in \cite{fujino1}, 
\cite[Section 12]{fujino3}, \cite{fujino-payne}, and so on. 
Finally, we prove 
the following theorem as an application 
of Theorem \ref{f-thm1.3} 
by using some basic properties of algebraic spaces. 

\begin{thm}[Nakai--Moishezon ampleness criterion 
for real line bundles on complete algebraic spaces]\label{f-thm1.6} 
Let $X$ be a complete algebraic space over an algebraically closed field 
and let 
$\mathcal L$ be an $\mathbb R$-line bundle 
on $X$. Then $\mathcal L$ is ample if and only if 
$\mathcal L^{\dim Z}\cdot Z>0$ for every 
positive-dimensional closed integral subspace $Z \subset X$. 
\end{thm}

We note that we treat algebraic spaces only in the final section, 
where we prove Theorem \ref{f-thm1.6}. 
We also note that the Nakai--Moishezon ampleness 
criterion for line bundles on complete algebraic spaces 
plays an crucial role in Koll\'ar's projectivity criterion for 
moduli spaces (see \cite{kollar} and \cite{fujino-moduli}). 
 
\begin{ack}
The first author was partially 
supported by JSPS KAKENHI Grant Numbers 
JP16H03925, JP16H06337. 
The second author was partially supported by 
JSPS KAKENHI Grant Number 20J20070. 
The authors thank Yoshinori Gongyo and Kenta Hashizume for 
comments. 
\end{ack}

\section{Preliminaries}\label{f-sec2}

For simplicity of notation, we write the group law of 
$\Pic(X)\otimes _{\mathbb Z}\mathbb R$ additively. 

\begin{defn}\label{f-def2.1}
Let $\mathcal L$ be an $\mathbb R$-line bundle on a complete 
scheme $X$. 
\begin{itemize}
\item If $\mathcal L=\sum _j l_j \mathcal L_j$ such that 
$l_j$ is a positive real number and $\mathcal L_j$ is an ample 
line bundle on $X$ for every $j$, then we say that $\mathcal L$ is {\em{ample}}. 
\item
If $\mathcal L=\sum _j l_j \mathcal L_j$ such that 
$l_j$ is a positive real number and $\mathcal L_j$ is a semi-ample 
line bundle on $X$ for every $j$, then we say that $\mathcal L$ 
is {\em{semi-ample}}. 
\item If $\mathcal L\cdot C\geq 0$ for every curve $C$ on $X$, 
then we say that $\mathcal L$ is {\em{nef}}. 
\item
We further assume that $X$ is a variety. 
If $\mathcal L=\sum _j l_j \mathcal L_j$ such that 
$l_j$ is a positive real number and $\mathcal L_j$ is a big 
line bundle on $X$ for every $j$, then we say that $\mathcal L$ is {\em{big}}. 
\end{itemize}
\end{defn}

In the theory of minimal models, we usually use $\mathbb R$-Cartier 
divisors. In this paper, we do not use $\mathbb R$-Weil divisors. 
We only use $\mathbb R$-Cartier divisors. 

\begin{defn}\label{f-def2.2}
Let $X$ be a complete scheme. 
We consider the following natural homomorphism 
$$
\psi\colon \Div(X)\otimes _{\mathbb Z}\mathbb R\to 
\Pic(X)\otimes _{\mathbb Z} \mathbb R. 
$$ We note that $\psi$ is not necessarily surjective. 
Let $D$ be an $\mathbb R$-Cartier divisor on $X$. 
If the image of $D$ by $\psi$ is ample, 
semi-ample, nef, and big, then 
$D$ is said to be {\em{ample}},  
{\em{semi-ample}}, {\em{nef}}, and {\em{big}}, respectively. 
We also note that $\psi$ is surjective when $X$ is 
a variety. 
\end{defn}

For the basic properties of bigness and 
semi-ampleness, see \cite[Section 2.1]{fujino2}. 
Here we only explain the following important characterization of 
nef and big $\mathbb R$-divisors. 
 
\begin{lem}\label{f-lem2.3}
Let $L$ be a nef $\mathbb R$-divisor on a projective 
variety $X$. 
Then $L$ is big if and only if $L^{\dim X}>0$. 
\end{lem}
\begin{proof}
We put $n=\dim X$. 
\begin{step}
If $L$ is big, then we can write $L\sim _{\mathbb R} A+D$ where 
$A$ is an ample $\mathbb R$-divisor 
and $D$ is an effective $\mathbb R$-Cartier divisor on $X$, 
where $\sim _{\mathbb R}$ denotes the $\mathbb R$-linear 
equivalence of $\mathbb R$-Cartier divisors. 
Then 
\begin{equation*}
L^n=(A+D)\cdot L^{n-1} \geq A\cdot L^{n-1}=A\cdot (A+D)\cdot L^{n-2} \geq \cdots 
\geq A^n>0. 
\end{equation*}
Hence $L^{\dim X}>0$ holds true when $L$ is big. 
\end{step}
\begin{step} 
In this step, we will 
check that $L$ is big under the assumption that $L^n>0$ holds. 
We will closely follow the proof of \cite[Theorem 2.3.18]{lazarsfeld}. 
We take ample $\mathbb R$-divisors $A_1$ and $A_2$ on $X$ such that 
$L+A_1$ and $A_1+A_2$ are $\mathbb Q$-Cartier divisors on $X$. 
Since $L^n>0$, we can assume that 
$$
(L+A_1)^n>n\left((L+A_1)^{n-1}\cdot (A_1+A_2)\right)
$$ 
holds by taking $A_1$ and $A_2$ sufficiently small. 
By the numerical criterion for bigness (see \cite[Theorem 2.2.15]{lazarsfeld}), 
$$
L-A_2=(L+A_1)-(A_1+A_2)
$$ 
is big. 
Hence $L$ is also big. 
\end{step}
We finish the proof of Lemma \ref{f-lem2.3}. 
\end{proof}

\section{Augmented base loci of $\mathbb R$-divisors}\label{f-sec3}

In this section, we explain some properties of augmented base loci of 
$\mathbb R$-divisors following \cite{birkar}. 
Let us start with the definition of base loci and stable base loci. 

\begin{defn}[Base loci and stable base loci of $\mathbb Q$-divisors]
\label{f-def3.1}
Let $X$ be a projective scheme and let $D$ be a Cartier divisor 
on $X$. 
The {\em{base locus}} of $D$ is defined as 
$$
\Bs\!|D|=\{ x\in X \,|\, \text{$\alpha$ vanishes at $x$ for every 
$\alpha \in H^0(X, \mathcal O_X(D))$}\}. 
$$ 
We consider $\Bs\!|D|$ with the reduced scheme structure. 

The {\em{stable base locus}} of a $\mathbb Q$-Cartier divisor 
$L$ on $X$ is defined as 
$$
\mathbf B(L)=\bigcap _m \Bs\!|mL|
$$ 
where $m$ runs over all positive integers 
such that $mL$ is Cartier. 
Note that $\mathbf B(L)$ is considered with the reduced scheme structure. 
\end{defn}

The notion of augmented base loci plays a crucial role 
in the theory of minimal models. 

\begin{defn}[Augmented base loci of $\mathbb R$-divisors]\label{f-def3.2}
Let $X$ be a projective scheme and let $L$ be an $\mathbb R$-Cartier 
divisor on $X$. 
We put 
$$
\mathbf B_+(L)=\bigcap _H \mathbf B(L-H)
$$ 
where $H$ runs over all ample $\mathbb R$-divisors such that 
$L-H$ is $\mathbb Q$-Cartier. 
As usual, we consider $\mathbf B_+(L)$ with 
the reduced scheme structure. 
We call $\mathbf B_+(L)$ the {\em{augmented base locus}} of 
$L$. 
\end{defn}

Birkar defined 
$\mathbf B_+(L)$ differently (see \cite[Definition 1,2]{birkar}). 
Then he proved that his definition coincides with the usual 
one (see Definition \ref{f-def3.2}). For the details, see 
\cite[Lemma 3.1 (3)]{birkar}. 

In order to explain Birkar's theorem (see Theorem \ref{f-thm3.4}), 
it is convenient to introduce 
the notion of exceptional loci of $\mathbb R$-divisors. 

\begin{defn}[Exceptional loci of $\mathbb R$-divisors]\label{f-def3.3} 
Let $X$ be a projective scheme and let $L$ be an $\mathbb R$-Cartier 
divisor on $X$. 
The {\em{exceptional locus}} of $L$ is 
defined as 
$$
\mathbb E(L)=\bigcup_{\dim V>0, \ L|_V \ \text{is not big}} V, 
$$ 
that is, the union runs over the positive-dimensional 
subvarieties $V\subset X$ such that 
$L|_V$ is not big. 
\end{defn}

Note that $\mathbb E(L)$ is sometimes called 
the {\em{null locus}} of $L$ when $L$ is nef. 

\begin{thm}[{\cite[Theorem 1.4]{birkar}}]\label{f-thm3.4}
Let $X$ be a projective scheme. 
Assume that $L$ is a nef $\mathbb R$-divisor on $X$. 
Then 
$$
\mathbf B_+(L)=\mathbb E(L)
$$ 
holds. 
\end{thm}

For the details of Theorem \ref{f-thm3.4}, 
we strongly recommend the reader to see Birkar's original 
statement in \cite[Theorem 1.4]{birkar}. 
We will use Theorem \ref{f-thm3.4} when 
$X$ is a normal projective variety in the proof of 
Theorem \ref{f-thm1.4}. 

\section{Proof of Theorem \ref{f-thm1.4}}\label{f-sec4}

In this section, we prove Theorem \ref{f-thm1.4}. 
The main ingredient of the proof of Theorem \ref{f-thm1.4} below 
is Birkar's theorem (see Theorem \ref{f-thm3.4}). 

\begin{proof}[Proof of Theorem \ref{f-thm1.4}]
Let $$
X=\bigcup _{i=1}^k U_i
$$ 
be a finite affine Zariski open cover of $X$. 
Let $\overline{U}_i$ be the closure of 
$U_i$ in $\mathbb P^{N_i}$. 
By \cite[Lemma 2.2]{lutkebohmert}, which is an easy application of 
the flattening theorem (see \cite[Th\'eor\`eme (5.2.2)]{raynaud}), 
we can take an ideal sheaf 
$\mathcal I$ on $\overline {U}_i$ 
with $\Supp \mathcal O_{\overline {U}_i}/\mathcal I\subset 
\overline {U}_i\setminus U_i$ 
such that the blow-up of $\overline {U}_i$ along 
$\mathcal I$ eliminates the indeterminacy of $\overline {U}_i
\dashrightarrow X$. 
Therefore, by taking the normalization of the blow-up of 
$\overline {U}_i$ along $\mathcal I$, we get a projective 
birational morphism $\pi_i\colon X_i\to X$ 
from a normal projective variety $X_i$ 
such that 
$\pi_i\colon \pi^{-1}_i(U_i)\to U_i$ is an isomorphism. 
$$
\xymatrix{
& X_i\ar[dl]\ar[dr]^-{\pi_i} & \\ 
\overline {U}_i \ar@{-->}[rr]&& X
}
$$ 
By Theorem \ref{f-thm3.4}, 
there exists an ample $\mathbb R$-divisor 
$H_i$ on $X_i$ such that 
$\pi^*_i L-H_i$ is $\mathbb Q$-Cartier 
and that 
$$
\mathbf B (\pi^*_i L-H_i)=\mathbf B_+(\pi^*_iL)=\mathbb E(\pi^*_iL)
$$ 
holds. Let $\Exc(\pi_i)$ be the exceptional locus of $\pi_i$. 
Then 
$$
\mathbb E(\pi^*_iL)=\Exc (\pi_i)
$$ holds by Lemma \ref{f-lem2.3} and 
the assumption that $L^{\dim Z}\cdot Z>0$ 
for every positive-dimensional closed subvariety $Z\subset X$. 
Since $L$ is $\mathbb R$-Cartier, 
we can write 
$$
L=\sum _{j\in J} l_j L_j
$$ 
such that $l_j$ is a real number and $L_j$ is Cartier 
for every $j\in J$. 
If $m_j \in \mathbb Q$ holds for every $j\in J$ and 
$$
\pi^*_i\left(\sum _{j\in J} m_jL_j\right)-\pi^*_i L+H_i
$$ 
is ample, then 
$$
\mathbf B\left(\pi^*_i \left(\sum _{j\in J} m_j L_j\right)\right)\subset 
\mathbf B(\pi^*_i L-H_i) =\Exc (\pi_i)
$$ 
holds. 
This implies that 
$$
\mathbf B\left(\sum _{j\in J} m_j L_j\right)\subset \pi_i (\Exc(\pi_i))
\subset X\setminus U_i. 
$$ 
Hence, there exists a positive real number $\varepsilon$ such that 
if $m_j\in \mathbb Q$ and $|l_j-m_j|<\varepsilon$ for every $j\in J$ then 
$$
\mathbf B\left(\sum _{j\in J} m_j L_j\right)\subset \bigcap _{i=1}^k 
\left(X\setminus U_i\right)=\emptyset
$$ 
holds. 
This means that $\sum _{j\in J}m_jL_j$ is semi-ample. 
By this fact, we can write 
$$
L=\sum_p r_p M_p 
$$ such that $r_p$ is a positive real number and 
$M_p$ is a semi-ample $\mathbb Q$-divisor for every $p$. 
Therefore, $L$ is a semi-ample $\mathbb R$-divisor by definition. 
Thus there exist a morphism $f\colon X\to Y$ onto a 
normal projective variety $Y$ with $f_*\mathcal O_X\simeq 
\mathcal O_Y$ and 
an ample $\mathbb R$-divisor $A$ on $Y$ such that 
$L$ is $\mathbb R$-linearly equivalent to 
$f^*A$ (see \cite[Lemma 2.1.11]{fujino2}). 
By assumption, $L\cdot C>0$ for every curve $C$ on 
$X$. This implies that $f$ is an isomorphism. 
Thus $L$ is an ample $\mathbb R$-divisor. 
\end{proof}

\section{Proof of Theorem \ref{f-thm1.3}}\label{f-sec5}

In this section, we prove Theorem \ref{f-thm1.3}. 
More precisely, we reduce Theorem \ref{f-thm1.3} to 
a special case where $X$ is a normal variety, which is 
nothing but Theorem \ref{f-thm1.4}. Let us start with the following 
elementary lemma. 

\begin{lem}\label{f-lem5.1}
Let $X$ be a complete scheme and let 
$\mathcal L$ be an $\mathbb R$-line bundle on $X$. 
Let $X=\bigcup_{i=1}^k X_i$ be the irreducible decomposition of $X$. 
Then $\mathcal L$ is ample if and only if 
$\mathcal L|_{(X_i)_{\mathrm{red}}}$ is ample 
for every $i$. 
\end{lem}
\begin{proof}
This statement is well known for $\mathbb Q$-line bundles. 
Hence we will freely use this lemma for $\mathbb Q$-line bundles 
in the following argument. 
\setcounter{step}{0}
\begin{step}
If $\mathcal L$ is ample, then 
it is obvious that so is $\mathcal L|_{(X_i)_{\mathrm{red}}}$ for 
every $i$. 
This is because we can write 
$$
\mathcal L=\sum _j a_j \mathcal L_j,  
$$ 
where $\mathcal L_j$ is an ample line bundle 
on $X$ and $a_j>0$ for every $j$. 
\end{step}
\begin{step}
In this step, we will prove that $\mathcal L$ is ample 
under the assumption that 
$\mathcal L|_{(X_i)_{\mathrm{red}}}$ is ample for every $i$. 
Since $\mathcal L$ is an $\mathbb R$-line bundle, we can write 
$$
\mathcal L=\sum _{j=1}^m l_j \mathcal L_j, 
$$ 
where $\mathcal L_j\in \Pic (X)$ and $l_j\in \mathbb R$ for 
every $j$. 
We put 
$$
V_i =\left\{ (p_1, \ldots, p_m)\in \mathbb R^m\, \left|\, 
\sum _{j=1}^m p_j \mathcal L_j|_{(X_i)_{\mathrm{red}}} 
\ \text{is ample}\right.\right\}
$$
for every $i$. Then $V_i$ contains an open neighborhood of 
$l=(l_1, \ldots, l_m)$ for every $i$ since 
$\sum _{j=1}^ml_j \mathcal L_j|_{(X_i)_{\mathrm{red}}}$ is ample 
by assumption. 
Hence $V=\bigcap _{i=1}^k V_i$ contains 
a small open neighborhood of $l\in \mathbb R^m$. 
Thus we can take positive real numbers $r_1, \ldots, r_p$ and 
$$
v_1=(v_{11}, \ldots, v_{1m}), \ \ldots, \ v_p =(v_{p1}, \ldots, v_{pm})\in V\cap 
\mathbb Q^m
$$ 
such that $l=\sum _{\alpha=1} ^p r_\alpha v_\alpha$. 
Then 
$$
\mathcal A_\alpha:=\sum _{j=1}^m v_{\alpha j} \mathcal L_j \in \Pic(X)\otimes 
_{\mathbb Z} \mathbb Q
$$ 
is ample for every $\alpha$ since $v_\alpha \in V\cap \mathbb Q^m$. 
Since we can write 
$$
\mathcal L=\sum _{\alpha=1}^p r_\alpha \mathcal A_\alpha, 
$$ 
$\mathcal L$ is ample by definition. 
\end{step}
We finish the proof of Lemma \ref{f-lem5.1}. 
\end{proof}

\begin{lem}\label{f-lem5.2}
Let $X$ be a complete variety 
and let $\mathcal L$ be an $\mathbb R$-line bundle on 
$X$. Let $\pi\colon Y\to X$ be a finite surjective morphism between 
complete varieties. 
Then $\mathcal L$ is ample if and only if $\pi^*\mathcal L$ is ample. 
\end{lem}

\begin{proof}
This statement is well known for $\mathbb Q$-line bundles. 
Hence we will freely use this lemma for $\mathbb Q$-line bundles 
in this proof. 
Thus it is obvious that 
$\pi^*\mathcal L$ is ample when $\mathcal L$ is ample. 
Therefore, it is sufficient to prove that $\mathcal L$ is ample 
under the assumption that 
$\pi^*\mathcal L$ is ample. 
Since $\mathcal L$ is an $\mathbb R$-line bundle, 
we can write 
$$
\mathcal L=\sum _{j=1}^m l_j \mathcal L_j, 
$$ 
where $\mathcal L_j \in \Pic (X)$ and $l_j\in \mathbb R$ for 
every $j$. 
Since $\pi^*\mathcal L$ is ample, 
there exists a positive real number $\varepsilon$ such that 
if $|l_j-\alpha_j|<\varepsilon$ for every $j$ then 
$$\pi^*\left(\sum _{j=1}^m \alpha_j \mathcal L_j\right)$$ is ample. 
Moreover, if we further assume 
$\alpha_j\in \mathbb Q$ for every $j$, 
then  $$\sum _{j=1}^m \alpha_j \mathcal L_j$$ is ample since 
$\pi$ is a finite surjective morphism. 
Hence we can write 
$$
\mathcal L=\sum _i r_i \mathcal A_i
$$ 
such that $r_i$ is a positive real number and $\mathcal A_i$ is an ample 
line bundle for every $i$. 
This means that $\mathcal L$ is ample by definition. 
\end{proof}

Let us prove Theorem \ref{f-thm1.3}. 

\begin{proof}[Proof of Theorem \ref{f-thm1.3}]
By Lemma \ref{f-lem5.1}, we may assume that $X$ is a 
variety. Let $\nu\colon X^\nu\to X$ be the normalization. 
Note that $\nu$ is a finite surjective morphism. 
Then by Lemma \ref{f-lem5.2} it is sufficient to 
prove that $\nu^*\mathcal L$ is ample. 
Hence we may further assume that $X$ is a complete normal 
variety. In this case, the ampleness of $\mathcal L$ follows 
from Theorem \ref{f-thm1.4}. 
\end{proof}

\section{Proof of Theorem \ref{f-thm1.5}}\label{f-sec6}

In this section, we prove Theorem \ref{f-thm1.5}. 
The following lemma is well known for $\mathbb Q$-line bundles. 

\begin{lem}\label{f-lem6.1}
Let $\pi\colon X\to S$ be a proper surjective morphism 
between schemes and let $\mathcal L$ be an $\mathbb R$-line 
bundle on $X$. 
Assume that $\mathcal L|_{X_s}$ is ample 
for every closed point $s\in S$, where $X_s=\pi^{-1}(s)$. 
Then 
$\mathcal L$ is $\pi$-ample, that is, we can write 
$$\mathcal L=\sum _i a_i \mathcal L_i$$ 
in 
$\Pic(X)\otimes _{\mathbb Z}\mathbb R$ such that 
$\mathcal L_i$ is a $\pi$-ample line bundle 
on $X$ and $a_i$ is a positive real number for every $i$. 
\end{lem}

Before we prove Lemma \ref{f-lem6.1}, we 
prepare the following lemma, which is also well known for 
$\mathbb Q$-line bundles. 

\begin{lem}\label{f-lem6.2}
Let $\pi\colon X\to S$ be a proper surjective morphism 
between schemes and let $\mathcal L$ be an $\mathbb R$-line 
bundle on $X$. 
Assume that $\mathcal L|_{X_{s_0}}$ is ample 
for some closed point $s_0\in S$, where $X_{s_0}=\pi^{-1}(s_0)$. 
Then there exists a Zariski open neighborhood $U_{s_0}$ of $s_0$ 
such that $\mathcal L|_{\pi^{-1}(U_{s_0})}$ is 
ample over $U_{s_0}$. 
\end{lem}

Although Lemmas \ref{f-lem6.1} and \ref{f-lem6.2} are 
more or less known to the experts, we can not 
find them in the standard literature. 
Hence we prove them here for the sake of completeness.

\begin{proof}[Proof of Lemma \ref{f-lem6.2}]
Since $\mathcal L$ is an $\mathbb R$-line bundle, 
there exist line bundles $\mathcal M_j$ for 
$1\leq j\leq k$ such that 
$$
\mathcal L=\sum _{j=1}^k b_j \mathcal M_j
$$ 
in $\Pic(X)\otimes _{\mathbb Z} \mathbb R$, where 
$b_j\in \mathbb R$ for 
every $j$. 
We put 
$$
\mathcal A=\left\{(c_1, \ldots, c_k)\in \mathbb R^k \left| \ \sum _{j=1}^k c_j 
\mathcal M_j|_{X_{s_0}} \ \text{is ample} \right.\right\}. 
$$ 
Then $\mathcal A$ contains a small open neighborhood 
of $(b_1, \ldots, b_k)$. 
Hence we can write 
$$
\mathcal L=\sum _i a_i \mathcal L_i, 
$$ 
where $\mathcal L_i$ is a line bundle on $X$ 
such that $a_i$ is a positive real number 
and $\mathcal L_i|_{X_{s_0}}$ is ample for every $i$. 
Since $\mathcal L_i|_{X_{s_0}}$ is ample for every $i$, 
there exists a Zariski open neighborhood 
$U_{s_0}$ of $s_0$ such that 
$\mathcal L_i|_{\pi^{-1}(U_{s_0})}$ is ample over $U_{s_0}$ for every $i$ 
(see, for example, \cite[Proposition 1.41]{kollar-mori}). 
Therefore, $\mathcal L|_{\pi^{-1}(U_{s_0})}=\sum _ia_i 
\mathcal L_i|_{\pi^{-1}(U_{s_0})}$ is ample 
over $U_{s_0}$. 
\end{proof}

Let us prove Lemma \ref{f-lem6.1}. 

\begin{proof}[Proof of Lemma \ref{f-lem6.1}]
We use the same notation as in the proof of 
Lemma \ref{f-lem6.2}. 
By Lemma \ref{f-lem6.2}, we can take 
$s_1, \ldots, s_l \in S$ such that 
$$
\bigcup _{\alpha=1}^l U_{s_\alpha}=S
$$ 
and 
that $\mathcal L|_{\pi^{-1}(U_{s_{\alpha}})}$ is ample 
over $U_{s_\alpha}$ for every $\alpha$. 
We put 
$$
\mathcal A_\alpha=\left\{(c_1, \ldots, c_k)\in \mathbb R^k\left| \ \sum _{j=1}^k c_j 
\mathcal M_j|_{\pi^{-1}(U_{s_{\alpha}})} \ \text{is $\pi$-ample 
over $U_{s_\alpha}$} \right.\right\}. 
$$ 
Then $\mathcal A_{\alpha}$ contains 
a small open neighborhood of 
$(b_1, \ldots, b_k)$. 
Therefore, $\bigcap _{\alpha=1}^l \mathcal A_{\alpha}$ also contains 
a small open neighborhood of $(b_1, \ldots, b_k)$. 
Hence we can write 
$$\mathcal L=\sum _i a_i \mathcal L_i$$ 
in 
$\Pic(X)\otimes _{\mathbb Z}\mathbb R$ such that 
$\mathcal L_i$ is a $\pi$-ample line bundle 
on $X$ and $a_i$ is a positive real number for every $i$. 
\end{proof}

Finally, we prove Theorem \ref{f-thm1.5}. 

\begin{proof}[Proof of Theorem \ref{f-thm1.5}]
If $\mathcal L$ is $\pi$-ample, then 
it is obvious that it satisfies 
the desired property. 
Hence, by Lemma \ref{f-lem6.1}, 
it is sufficient to prove that $\mathcal L|_{X_s}$ is ample 
for every closed point $s\in S$, where 
$X_s=\pi^{-1}(s)$, 
under the assumption that $\mathcal L^{\dim Z}\cdot Z>0$. 
This follows from the Nakai--Moishezon ampleness criterion for 
$\mathbb R$-line bundles on complete schemes (see Theorem \ref{f-thm1.3}). 
\end{proof}

\section{Proof of Theorem \ref{f-thm1.6}}\label{f-sec7}

In this final section, we prove the Nakai--Moishezon 
ampleness criterion for $\mathbb R$-line bundles on 
complete algebraic spaces (Theorem \ref{f-thm1.6}) 
by using some basic properties of 
algebraic spaces and Theorem \ref{f-thm1.3}. 

\begin{proof}[Proof of Theorem \ref{f-thm1.6}] 
It is well known that the Nakai--Moishezon ampleness 
criterion holds for line bundles on complete 
algebraic spaces (see, for example, \cite[3.11.~Theorem]{kollar} 
and \cite[(1.4) Theorem]{p}). 
It is also well known that there exists a finite surjective morphism 
$f\colon Y\to X$ from a complete scheme $Y$ (see, 
for example, \cite[2.8.~Lemma]{kollar}). 
By Theorem \ref{f-thm1.3}, 
$f^*\mathcal L$ is an ample $\mathbb R$-line bundle 
on $Y$. 
We write 
$$
\mathcal L=\sum _j a_j \mathcal L_j, 
$$ 
where $\mathcal L_j$ is a line bundle on $X$ and $a_j$ is a real 
number for every $j$. 
We put 
$$
\mathcal M=\sum _j b_j \mathcal L_j, 
$$ 
where $b_j$ is a rational number for every $j$. 
If $|a_j-b_j|\ll 1$ for every $j$, 
then $f^*\mathcal M$ is an ample $\mathbb Q$-line bundle 
on $Y$ since $f^*\mathcal L$ is ample. 
Therefore, $m\mathcal M$ is an ample line bundle on $X$ for 
some positive integer $m$ by the 
Nakai--Moishezon ampleness criterion for line bundles 
on complete algebraic spaces. 
This implies that $X$ is projective. 
Thus, by Theorem \ref{f-thm1.3} again, $\mathcal L$ is 
an ample $\mathbb R$-line bundle on $X$. 
\end{proof}

%%%%%%%%%%%%%%%


\begin{thebibliography}{KoM} 

\bibitem[B]{birkar}
C.~Birkar, The augmented base locus of real 
divisors over arbitrary fields, 
Math. Ann. \textbf{368} (2017), no. 3-4, 905--921. 

\bibitem[CP]{campana-peternell} 
F.~Campana, T.~Peternell, 
Algebraicity of the ample cone of projective 
varieties, 
J. Reine Angew. Math. \textbf{407} (1990), 160--166.

\bibitem[F1]{fujino1} 
O.~Fujino, On the Kleiman--Mori cone, 
Proc. Japan 
Acad. Ser. A Math. Sci. \textbf{81} (2005), no. 5, 80--84.

\bibitem[F2]{fujino2}
O.~Fujino, {\em{Foundations of the minimal model program}}, 
MSJ Memoirs, \textbf{35}. Mathematical Society of Japan, Tokyo, 2017. 

\bibitem[F3]{fujino-moduli} 
O.~Fujino, Semipositivity theorems for moduli problems, 
Ann. of Math. (2) \textbf{187} (2018), no. 3, 639--665. 

\bibitem[F4]{fujino3} 
O.~Fujino, 
Minimal model theory for log surfaces in Fujiki's class $\mathcal C$, 
to appear in Nagoya Math. J. 

\bibitem[FP]{fujino-payne}
O.~Fujino, S.~Payne, 
Smooth complete toric threefolds with no nontrivial nef line bundles, 
Proc. Japan Acad. Ser. A Math. Sci. \textbf{81} (2005), no. 10, 174--179. 

\bibitem[Kl]{kleiman} 
S.~L.~Kleiman, 
Toward a numerical theory of ampleness, 
Ann. of Math. (2) \textbf{84} (1966), 293--344.

\bibitem[Ko]{kollar}
J.~Koll\'ar, 
Projectivity of complete moduli, J. Differential Geom. \textbf{32} 
(1990), no. 1, 235--268.

\bibitem[KoM]{kollar-mori} 
J.~Koll\'ar, S.~Mori, {\em{Birational geometry of algebraic 
varieties}}. With the collaboration of C.~H.~Clemens 
and A.~Corti. Translated 
from the 1998 Japanese original. Cambridge 
Tracts in Mathematics, \textbf{134}. Cambridge 
University Press, Cambridge, 1998.

\bibitem[La]{lazarsfeld} 
R.~Lazarsfeld, {\em{Positivity in algebraic geometry. I. Classical setting: line 
bundles and linear series}}, 
Ergebnisse der Mathematik und ihrer Grenzgebiete. 3. Folge. A 
Series of Modern Surveys in Mathematics [Results in 
Mathematics and Related Areas. 3rd Series. A Series of 
Modern Surveys in Mathematics], \textbf{48}. Springer-Verlag, Berlin, 2004.

\bibitem[L\"u]{lutkebohmert}
W. L\"utkebohmert, On compactification of schemes, 
Manuscripta Math. \textbf{80} (1993), no. 1, 95--111. 

\bibitem[P]{p}
P.~Pascual Gainza, 
Ampleness criteria for algebraic spaces, 
Arch. Math. (Basel) \textbf{45} (1985), no. 3, 270--274.

\bibitem[RG]{raynaud} 
M.~Raynaud, L.~Gruson, Crit\`eres de 
platitude et de projectivit\`e.~Techniques de \lq\lq platification\rq\rq d'un 
module, 
Invent. Math. \textbf{13} (1971), 1--89.
\end{thebibliography}
\end{document}